\documentclass{paper}
\usepackage{amsmath}
  \usepackage[table]{xcolor}
\usepackage{array,hhline}
\usepackage{amsthm}
\usepackage{amsfonts}
\usepackage{amssymb}
\usepackage[english]{babel}
\usepackage{hyperref}

\newtheorem{theorem}{Theorem}[section]
\newtheorem{corollary}[theorem]{Corollary}

\newtheorem{lemma}[theorem]{Lemma}
\newtheorem{proposition}[theorem]{Proposition}
\theoremstyle{definition}
\newtheorem{remark}{Remark}

\begin{document}

\title{\boldmath On the number of frequency hypercubes~F$^n(4;2,2)$%
\thanks{The work of M. J. Shi is supported by
the National Natural Science Foundation of China (12071001, 61672036),
the Excellent Youth Foundation of Natural Science Fo12071001undation of Anhui Province (1808085J20),
the Academic Fund for Outstanding Talents in Universities (gxbjZD03);
the work of D. S. Krotov is
supported within the framework of the state contract of the Sobolev Institute of Mathematics (project no. 0314-2019-0016).}
}
\author{Minjia~Shi\thanks{School of Mathematical Sciences, Anhui University, Hefei, 230601, China}, Shukai Wang, Xiaoxiao Li,
and {Denis~S.~Krotov}\thanks{Sobolev Institute of Mathematics, Novosibirsk, 630090, Russia}
}
\maketitle

\begin{abstract}
A {frequency  $n$-cube}
F$^n(4;2,2)$
is an $n$-dimensional
$4\times\ldots\times 4$ array filled by $0$s and $1$s
such that each line contains exactly two $1$s.
We classify the frequency  $4$-cubes F$^4(4;2,2)$,
find a testing set of size $25$ for F$^3(4;2,2)$,
and derive an upper bound on the number of F$^n(4;2,2)$.
Additionally, for any $n$ greater than $2$, we construct an
F$^n(4;2,2)$ that cannot be refined to a latin hypercube,
while each of its sub-F$^{n-1}(4;2,2)$ can.

\end{abstract}

\section{Introduction}

A \emph{frequency hypercube} (\emph{frequency $n$-cube})
F$^n(q;\lambda_0, \ldots , \lambda_{m-1})$
is an $n$-dimensional
$q\times \cdots \times q$ array, where $q=\lambda_0+ \cdots +\lambda_{m-1}$,
filled by numbers $0, \ldots, m-1$
with the property that each line contains exactly $\lambda_i$ cells with symbol $i$, $i=0,\ldots , m-1$
(a \emph{line} consists of $q$ cells of the array differing in one coordinate).
Frequency hypercubes generalize both frequency squares, $n=2$,
and latin hypercubes, $\lambda_0= \cdots =\lambda_{q-1}=1$
(alternatively, frequency hypercubes F$^{n+1}(q;1,q-1)$
are also equivalent to latin $n$-cubes),
see respectively
\cite{BraMul:86},
\cite{Krc:2007}
and
\cite{MK-W:small}
for the enumeration of the objects of small size.
In their turn, frequency squares and  latin hypercubes
are different generalizations of latin squares,
which have parameters F$^2(q;1, \ldots ,1)$ in our notation.

In the current work, we study frequency hypercubes of order $q=4$.
Latin hypercubes of order $4$,
corresponding to
F$^n(4;1,1,1,1)$
and F$^n(4;1,3)$
are characterized in~\cite{KroPot:4}, and their
structures are now rather well understood;
in particular,
not only the asymptotics~\cite{PotKro:asymp},
but also an effective recursive
formula
for their number is known~\cite{PotKro:numberQua.en},
see sequences~A211214 and~A211215 in OEIS~\cite{OEIS}.
Another case, F$^n(4;1,1,2)$, is solved
in~\cite{Pot:2012:partial}, where it is proved that
every frequency $n$-cube F$^n(4;1,1,2)$
can be subdivided to F$^n(4;1,1,1,1)$.
We focus on the remaining case F$^n(4;2,2)$.
It should also be mentioned that all frequency hypercubes
with $q<4$ correspond to latin hypercubes of order $q$,
which form only one isotopy class, for each $n$ and $q$.

The goals of our paper are, firstly, to enumerate all frequency hypercubes
F$^3(4;2,2)$ and F$^4(4;2,2)$, and secondly,
improve the upper bound on the number of frequency hypercubes
F$^n(4;2,2)$. The second result is based on calculations
in small dimensions,
where we study some more general class of objects than the frequency
hypercubes, called unitrades.
We use the language of coding theory as the working language in this paper and
study sets of vertices of the Hamming graph, called codes. In particular,
the frequency hypercubes F$^n(4;2,2)$
are represented by the special
codes called double-MDS-codes in the Hamming graph $H(n,4)$
in such a way that a frequency hypercube F$^n(4;2,2)$ is
the characteristic function of a double-MDS-code.

The trivial upper bound $2^{3^n}$ on the number of frequency cubes F$^n(4;2,2)$
is based on the existence of an easy-to-find testing set from $3^n$ points 
such that the frequency cube is uniquely defined by its values on those points.
Our new upper bound, of form $2^{\alpha^n}$, is obtained by finding an appropriate testing set of size $\alpha^n$, where $\alpha<3$.
Note that a lower bound of form $2^{2^{n+o(1)}}$ on the number of F$^n(4;2,2)$ is
immediate from the similar lower bound for F$^n(4;1,1,1,1)$
(see \cite{PotKro:asymp}). So, the number of objects is double-exponential.
In the framework of our study,
the following questions are actual and remain open:
What is the asymptotics of the double-logarithm of the number of  F$^n(4;2,2)$
(as well as F$^n(q;\lambda_0, \ldots , \lambda_{m-1})$ for any fixed
$q>4$, $\lambda_0$,\ldots , $\lambda_{m-1}$)?
What is the minimum size of a testing set for F$^n(4;2,2)$
(in general, for F$^n(q;\lambda_0, \ldots , \lambda_{m-1})$, $q>3$), in particular, for
F$^3(4;2,2)$?

As follows from the definition, double-MDS-codes are equivalent to
simple orthogonal arrays OA$(4^n/2,n,4,n-1)$,
see~\cite{GreCol:OA} for the definition and notation.
In~\cite{Potapov2019:medium}, based on a lower bound on the number of double-MDS-codes,
Potapov derived a lower bound on the number
of $n$-ary balanced correlation immune (resilient) Boolean functions of order $n/2$
(equivalently, simple orthogonal arrays OA$(2^{n-1},n,2,n/2)$).
Substituting the results of our computation of the number of F$^4(4;2,2)$
to the arguments in~\cite{Potapov2019:medium} straightforwardly
improves the bound~\cite{Potapov2019:medium} 
for correlation immune functions for $n=8$.

\smallskip

The material is organized as follows. Section~\ref{sec:2} sets up the basic notations and definitions.
Section~\ref{333} gives a straightforward algorithm to classify double-MDS-codes and unitrades.
The results obtained by implementation of the algorithm are shown in tables.
Section~\ref{s:split} introduces the splittability property of double-MDS-codes.
The main result of this paper is established in Section~\ref{sec:5}.
We give a testing set of size $25$ for F$^3(4;2,2)$
and derive an upper bound on the number of F$^n(4;2,2)$.


\section{Basic notations and definitions}\label{sec:2}
\subsection{Double-codes and related objects}
Let $\Sigma=\{0,1, \ldots, q-1\}$ and $\Sigma^{n}$ be the set of words of length $n$ over the alphabet~$\Sigma$.
Subsets of $\Sigma^{n}$ are called codes of \emph{length} $n$.
We set $q=4$ in this paper,
while the definitions are valid for arbitrary $q$.
The \emph{Hamming graph} $H(n,q)$ is a graph on the vertex set $\Sigma^{n}$,
where two vertices are adjacent if they differ in one coordinate.
We call a set of $q$ pairwise adjacent vertices in $H(n,q)$ a \emph{line} (essentially, a line is a maximal clique in $H(n,q)$).
A (distance-$2$) \emph{MDS code} is a subset of $\Sigma^{n}$ intersecting every line in exactly one element
(we consider only the minimum-distance-$2$ MDS codes,
while in coding theory
this concept is more general).
A set $S\subset \Sigma^{n}$ is called a \emph{double-MDS-code} if each line of $\Sigma^{n}$ contains exactly two elements from $S$.
A set $S\subset \Sigma^{n}$
is called a \emph{unitrade}
if each line of $\Sigma^{n}$
contains an even number of elements from $S$.
If each line of $\Sigma^{n}$
contains zero or two  elements
from $S$,
then $S$ is called a \emph{double-code}.
A double-code is called \emph{splittable}
if it can be represented as a union of two independent sets
(if the double-code is a double-MDS-code, then these two sets are MDS codes).
The characteristic function
$\chi_S:\Sigma^n \to \{0,1\}$ of a double-MDS-code is
a frequency $n$-cube F$^n(q;q-2,2)$.

For example, the picture below illustrates
an incomplete collection of
unitrades in $H(2,4)$;
the first four of them (and the complement of the $5$th) are double-codes,
while only the $3$rd one and the $4$th one are double-MDS-codes, both splittable.
$$
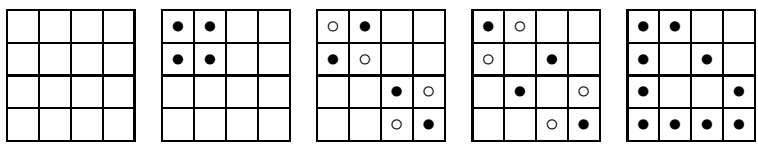
\begin{array}{|@{\ }c@{\ }|@{\ }c@{\ }|@{\ }c@{\ }|@{\ }c@{\ }|}
\hline
\phantom\bullet&\phantom\bullet&\phantom\bullet&\phantom\bullet\\
\hline
 & & & \\
\hline
 & & & \\
\hline
 & & & \\
\hline
\end{array}\quad
\begin{array}{|@{\ }c@{\ }|@{\ }c@{\ }|@{\ }c@{\ }|@{\ }c@{\ }|}
\hline
\bullet&\bullet&\phantom\bullet&\phantom\bullet\\
\hline
\bullet&\bullet&\ &\ \\
\hline
 & & & \\
\hline
 & & & \\
\hline
\end{array}\quad
\begin{array}{|@{\ }c@{\ }|@{\ }c@{\ }|@{\ }c@{\ }|@{\ }c@{\ }|}
\hline
\circ&\bullet& & \\
\hline
\bullet&\circ& & \\
\hline
 & &\bullet & \circ \\
\hline
 & &\circ&\bullet\\
\hline
\end{array}\quad
\begin{array}{|@{\ }c@{\ }|@{\ }c@{\ }|@{\ }c@{\ }|@{\ }c@{\ }|}
\hline
\bullet&\circ& & \\
\hline
\circ& &\bullet& \\
\hline
 &\bullet& &\circ  \\
\hline
 & &\circ&\bullet \\
\hline
\end{array}\quad
\begin{array}{|@{\ }c@{\ }|@{\ }c@{\ }|@{\ }c@{\ }|@{\ }c@{\ }|}
\hline
\bullet&\bullet& & \\
\hline
\bullet& &\bullet& \\
\hline
 \bullet&& &\bullet\\
\hline
 \bullet& \bullet&\bullet&\bullet\\
\hline
\end{array}
$$

\begin{remark}
 The concept ``unitrade'' was introduced in~\cite{Potapov:2013:trade} in the context
 of study of MDS codes. 
 The reader can observe a difference between our definition
 and that in~\cite{Potapov:2013:trade}.
 The reason is that the definition of a unitrade is based
 on the properties of the symmetric difference of two objects
 of the main class under the study. Our main class is double-MDS-codes,
 and the symmetric difference can have $0$, $2$, or $4$ elements 
 in a line; this property is taken as the definition of a unitrade.
 In~\cite{Potapov:2013:trade}, the main class is MDS codes, and a unitrade
 is allowed to have only $0$ or $2$ elements in a line, as the symmetric
 difference of any two MDS codes.
\end{remark}

Given a set $M \subset \Sigma^n$, a \emph{layer} of $M$ consists of
all elements of $M$ that have some fixed value $a$ in a given coordinate $i$
(referred to as the \emph{direction} of the layer):
$\{(c_0,\ldots,c_n)\in C \mid  c_i=a \}$.
We identify a layer with the corresponding subset of $\Sigma^{n-1}$
by ignoring the fixed coordinate.
Note that any layer inherits the property of the set to be
an MDS code, double-MDS-code, double-code, or unitrade.

\subsection{Isotopy and equivalence}
An \emph{isotopy} in $\Sigma^{n}$ is an $n$-tuple of permutations
$\theta_{i}:\Sigma\rightarrow\Sigma$, $i\in\{1,\ldots,n\}$.
For an isotopy $\bar\theta=(\theta_1,\ldots,\theta_n)$
and a set $S\subset\Sigma^{n}$,
we put
$$\bar{\theta}S=\{(\theta_{1}x_{1},\ldots,\theta_{n}x_{n})\mid (x_{1},\ldots,x_{n})\in S\}.$$
Two sets $S_{1}\subset\Sigma^{n}$
and $S_{2}\subset\Sigma^{n}$
are \emph{isotopic}
if there exists an isotopy
$\bar{\theta}$ such that
$\bar{\theta}S_{1}=S_{2}$.
An \emph{autotopy} of a set $S\subset\Sigma^{n}$
is an isotopy $\bar\theta$ such that
$\bar\theta S = S$.

Two sets $S_{1}\subset\Sigma^{n}$
and $S_{2}\subset\Sigma^{n}$
are \emph{equivalent} if there exists an isotopy
$\bar{\theta}$
and a permutation $\sigma$ of $n$ coordinates such that
$\sigma\bar{\theta}S_{1}=S_{2}$, where
$$
 \sigma S=\{(x_{\sigma^{-1}(1)},\ldots,x_{\sigma^{-1}(n)}) \mid
 (x_{1},\ldots,x_{n})\in S\}.
$$
If $\sigma\bar{\theta}S=S$ holds
for some set $S\subset\Sigma^{n}$, isotopy
$\bar{\theta}$ and coordinate permutation $\sigma$,
then the pair $(\sigma,\bar{\theta})$ is called
an \emph{automorphism} of $S$.
The group of all automorphisms (autotopies) of a set $S$
with the composition as the group operation
is denoted $\mathrm{Aut}(S)$ (respectively $\mathrm{Atop}(S)$).

Two frequency $n$-cubes $\chi_{S_1}$ and $\chi_{S_2}$
are \emph{isotopic} (\emph{equivalent}) if $S_1$ is isotopic (equivalent)
to $S_2$ or $\Sigma^n\backslash S_2$.
So, we see that the number of equivalence classes
of frequency $n$-cubes F$^n(4;2,2)$ is
in general smaller than that of double-MDS-codes in $H(n,4)$:
a double-MDS-code can be nonequivalent to its complement,
but the corresponding two frequency $n$-cubes are equivalent
by the definition. An \emph{autotopy} of a frequency $n$-cube $\chi_{S}$
is an isotopy that sends the corresponding double-MDS-code $S$ to itself or to its complement $\Sigma^n\backslash S$.
The \emph{automorphisms} of a frequency $n$-cube are defined similarly.
So, the set of autotopies (automorphisms) of a frequency
$n$-cube F$^n(4;2,2)$
either coincides with the set of autotopies (automorphisms) of
the corresponding double-MDS-code, or is twice larger than it.


\subsection{Testing sets}
A subset $T$ of $\Sigma^{n}$  is a
\emph{testing set} for double-MDS-codes
(or, equivalently, for any other class of subsets of $\Sigma^{n}$)
if $C_{1}\cap{T}\ne C_{2}\cap{T}$ for any two
different double-MDS-codes in $H(n,4)$.
\begin{table}[!htpb]
$$
\begin{array}{llllllllllllll}
\cline{1-4}
\multicolumn{1}{|l|}{} &
\multicolumn{1}{l|}{} &
\multicolumn{1}{l|}{}  &
\multicolumn{1}{l|}{\phantom{\bullet}} \\ \cline{1-4}
\multicolumn{1}{|l|}{\bullet} &
\multicolumn{1}{l|}{\bullet} &
\multicolumn{1}{l|}{\bullet}  &
\multicolumn{1}{l|}{} \\ \cline{1-4}
\multicolumn{1}{|l|}{\bullet} &
\multicolumn{1}{l|}{\bullet} &
\multicolumn{1}{l|}{\bullet}  &
\multicolumn{1}{l|}{} \\ \cline{1-4}
\multicolumn{1}{|l|}{\bullet} &
\multicolumn{1}{l|}{\bullet} &
\multicolumn{1}{l|}{\bullet}  &
\multicolumn{1}{l|}{} \\ \cline{1-4}
\end{array}\ \ \
\begin{array}{llllllllllllll}
\hhline{----}
\multicolumn{1}{|l|}{\cellcolor[HTML]{EEEEEE}} &
\multicolumn{1}{l|}{\cellcolor[HTML]{E5E5E5}} &
\multicolumn{1}{l|}{\cellcolor[HTML]{E5E5E5} \circ }  &
\multicolumn{1}{l|}{\cellcolor[HTML]{BBBBBB} \circ } \\ \hhline{----}
\multicolumn{1}{|l|}{\cellcolor[HTML]{FFFFFF} } &
\multicolumn{1}{l|}{\cellcolor[HTML]{FFFFFF}\bullet} &
\multicolumn{1}{l|}{\cellcolor[HTML]{FFFFFF} }  &
\multicolumn{1}{l|}{\cellcolor[HTML]{E5E5E5} \circ } \\ \hhline{----}
\multicolumn{1}{|l|}{\cellcolor[HTML]{FFFFFF}\bullet} &
\multicolumn{1}{l|}{\cellcolor[HTML]{FFFFFF} } &
\multicolumn{1}{l|}{\cellcolor[HTML]{FFFFFF}\bullet}  &
\multicolumn{1}{l|}{\cellcolor[HTML]{E5E5E5}} \\ \hhline{----}
\multicolumn{1}{|l|}{\cellcolor[HTML]{FFFFFF}\bullet} &
\multicolumn{1}{l|}{\cellcolor[HTML]{FFFFFF}\bullet} &
\multicolumn{1}{l|}{\cellcolor[HTML]{FFFFFF} }  &
\multicolumn{1}{l|}{\cellcolor[HTML]{E5E5E5}} \\ \hhline{----}
\end{array}
$$
\caption{A double-MDS-code (right table) is reconstructed by its intersection with the testing set (left table)}\label{t:test}
\end{table}
\par
For example, $\{0,1,2\}^n$ (see, e.g., Table~\ref{t:test}) is a testing set for double-MDS-codes in $H(n,4)$; its size is $3^n$, and it is not difficult to show
that it is minimum for $n=1$, $2$. However, we can find
a smaller testing set for $n=3$; the minimum size of a testing set
in the three-dimensional case is still unknown.

\section{Classification}\label{333}

Using a computational approach, we classify double-MDS-codes
in $H(3,4)$ and $H(4,4)$ and unitrades in $H(3,4)$.
The algorithm is rather straightforward. 
Below, we describe it for double-MDS-codes,
while for unitrades it is absolutely the same
(just replace ``double-MDS-code'' by ``unitrade'' everywhere).
Before we give the details of the algorithm, 
we say few words about recognizing the equivalence.
A usual way to work with the equivalence of codes is 
to represent them by graphs in such a way that
two codes are equivalent if and only if the corresponding  
graphs are isomorphic, see~\cite[\S3.3.2]{KO:alg}. 
A standard software that helps to recognize the graph isomorphism
is \texttt{nauty\&traces}~\cite{nauty2014};
it is realized as a package that can be used in \texttt{c} or \texttt{c++}
programs. With this package, for a graph one can compute its 
canonically-labelled version, such that two graphs are isomorphic
if and only if the corresponding canonically-labelled graphs are equal to each other.
The same procedure computes the automorphism group of the graph,
which can be used for the numerical validation of the results,
see Section~\ref{s:valid}.

\subsection{Algorithm}\label{333.1}

Assume that, as a result of the past classification,
we have that the number of
double-MDS-codes of length $n-1$ is $N_{n-1}$ and they are partitioned into $a$ equivalence
classes.

We consider a double-MDS-code of length $n$ as the union of
$4$ of its layers in the last direction, essentially, $4$ double-MDS-codes of length $n-1$.
As the fourth layer is uniquely reconstructed from the first three ones, we
can run through $N_{n-1}^3$ triples of double-MDS-codes of length $n-1$
and check
if they can be completed by the fourth layer
to form a double-MDS-code of length $n$.
In such a way, we construct all double-MDS-codes of length $n$.
However, simple calculations show that the number $N_{n-1}^3$, $n=4$, 
is too large to manage all these possibilities. 
A standard way to make the search faster is, at each step,
to choose only nonequivalent cases (this procedure is often called \emph{isomorph rejection}).

\label{sss}
At \textbf{step 1},
we choose the first layer from $a$
nonequivalent double-MDS-codes of length $n-1$
(representatives of the equivalence classes).

At \textbf{step 2},
for each choice of the first layer,
we choose the second layer from all $N_{n-1}$ double-MDS-codes of length $n-1$.
So, we process $a\cdot N_{n-1}$ cases at this step.
Among all possible unions of the first layer and the second layer,
called \emph{semi-codes},
we choose only nonequivalent $a'$ representatives,
where $a'$ is the number of the equivalence classes of semi-codes.

At \textbf{step 3}, for each of $a'$
representatives of semi-codes,
we choose the third layer from $N_{n-1}$
different double-MDS-codes of length $n-1$.
So, we process $a'\cdot N_{n-1}$ cases at this step.
Some of the resulting cases cannot be completed to a
double-MDS-code of length $n$
(for example, if there are already three codewords in some line of the last direction);
for the remaining cases, the completion is unique, and we use 
isomorph rejection to keep only nonequivalent representatives of the resulting double-MDS-codes of length $n$.

\subsection{Validation} \label{s:valid}
We can check the results of the classification
using a general double-counting approach, described in~\cite[\S10.2]{KO:alg}.
We know that the number of semi-codes is $N_{n-1}^2$,
and at step~2 we found representatives of their $a'$ equivalence classes.
The number of continuations to a double-MDS-code of length $n$
is the same for all semi-codes from the same, say $i$th, equivalence class.
We denote this number by $R_i$ and remark that it was found, for each $i$, at step~3.
Hence, the total number $N_n$ of double-MDS-codes of length $n$ 
can be found as
\begin{equation}\label{eq:Nsum}
N_n=\sum_{i=1}^{a'} M_i R_i ,
\end{equation}
where $M_i$ is the number of semi-codes
in the $i$th equivalence class
(it can be easily counted from the automorphism group
of the corresponding representative).
The number $N_n$ can be alternatively found as
\begin{equation}\label{eq:Nsum'}
N_n=\sum_{j}^{} T_j , \qquad T_j=\frac{4!^n\times n!}{|\mathrm{Aut}(C_j)|},
\end{equation}
where $T_j$ is the size of the $j$th equivalence class
of double-MDS-codes of length $n$ and $C_j$
is its representative (which is kept after the final isomorph rejection at step~3).
Coincidence of the values of $N_n$ obtained from \eqref{eq:Nsum} and \eqref{eq:Nsum'}
certifies that our algorithm gives correct numbers $R_i$ and, in particular, it did not miss anything.

\subsection{Results of the classification}

The results of the classification by the algorithm in Section \ref{333.1} are reflected in following theorems and
Tables~\ref{t:uni}, ~\ref{t:4}, and ~\ref{t:3}.
The database with representatives of equivalence classes can be found in~\cite{data}.

\begin{theorem}\label{th:M3}
 There are $51678$ double-MDS-codes of length $3$,
 divided into $10$ equivalence classes or $26$ isotopy classes.  There are $51678$ frequency F$^3(4;2,2)$ cubes,
 divided into $10$ equivalence classes or $26$ isotopy classes.
\end{theorem}

\begin{theorem}\label{th:M4}
 There are $ 55720396530 $ double-MDS-codes of length $4$,
 divided into $8895$ equivalence classes or  $ 192214 $ isotopy classes.  There are $ 55720396530 $ frequency F$^4(4;2,2)$ hypercubes,
 divided into $7203$
 equivalence classes or $154078$ isotopy classes.
\end{theorem}

\begin{theorem}\label{th:U3}
 There are $ 2^{27} $ unitrades in $H(3,4)$,
 divided into $ 2528 $ equivalence classes.
 Representatives of $312$ of them are equivalent to their complements.
\end{theorem}

As we will see below (Lemma~\ref{l:3n}), the total number 
of unitrades in $H(n,4)$ is $2^{3^n}$ for any $n$.

\begin{table}[htbp]
\centering
\begin{tabular}{|c|c|c|}
\hline
unitrade& \multicolumn{2}{|c|}{the number of equivalence classes}\\
\hhline{|~|-|-|}
size  & of unitrades & of double-codes \\
\hhline{|=|=|=|}
0 \text{ or } 64&1 & 1\\
\hline
8 \text{ or } 56&1 & 1\\
\hline
12 \text{ or } 52&1 & 1\\
\hline
14 \text{ or } 50&1 & 1\\
\hline
16 \text{ or } 48&9 & 6\\
\hline
18 \text{ or } 46&5 & 4\\
\hline
20 \text{ or } 44&22 & 11\\
\hline
22 \text{ or } 42&26 & 11\\
\hline
24 \text{ or } 40&121 & 29\\
\hline
26 \text{ or } 38&125 & 11\\
\hline
28 \text{ or } 36&329 & 15\\
\hline
30 \text{ or } 34&328 & 4\\
\hline
32               &590 & 10\\
\hhline{|=|=|=|}
total: & 2528 & 200\\
\hline
\end{tabular}
\caption[Unitrades in H(3,4)]{Unitrades in $H(3,4)$}\label{t:uni}
\end{table}

\begin{table}[htbp]
\centering
$\begin{array}{|r|r|r|r|r|}
\hline
 |\mathrm{Aut}(C)|   &N &N'&N''&N^*\\
\hhline{|=|=|=|=|=|}
 24\cdot 2048  & 1    &  1    & 1    &  1        \\ \hline
 4'\cdot 512   & 1    &  1    & 1    &  1        \\ \hline
 6\cdot 256    & 1    &  1    & 1    &  1        \\ \hline
 2'\cdot 256   & 1    &  1    & 1    &  1        \\ \hline
 24\cdot 128   & 1    &  1    & 1    &  1        \\ \hline
 8\cdot 128    & 2    &  2    & 2    &  2        \\ \hline
 6\cdot 128    & 1    &  1    & 1    &  1        \\ \hline
 4'\cdot 128   & 1    &  1    & 1    &  1        \\ \hline
 2'\cdot 128   & 2    &  2    & 2    &  2        \\ \hline
 2''\cdot 128  & 1    &  1    & 1    &  1        \\ \hline
 6\cdot 64     & 1    &  1    & 1    &  1        \\ \hline
 4'\cdot 64    & 2    &  2    & 2    &  2        \\ \hline
 2'\cdot 64    & 3    &  3    & 3    &  2        \\ \hline
 1\cdot 64     & 1    &  1    & 1    &           \\ \hline
 8\cdot 32     & 1    &  1    & 1    &  1        \\ \hline
 6\cdot 32     & 1    &  1    & 1    &  1        \\ \hline
 4'\cdot 32    & 3    &  3    & 3    &           \\ \hline
 2'\cdot 32    & 22   &  22   & 22   &  1        \\ \hline
 1\cdot 32     & 7    &  7    & 7    &           \\ \hline
 24\cdot 16    & 2    &  2    & 2    &  2        \\ \hline
 6\cdot 16     & 4    &  4    & 4    &  3        \\ \hline
 4'\cdot 16    & 8    &  8    & 8    &  7        \\ \hline
 3\cdot 16     & 1    &  1    & 0    &           \\ \hline
 2'\cdot 16    & 27   &  27   & 25   &  6        \\ \hline
 2''\cdot 16   & 4    &  4    & 2    &           \\ \hline
 1\cdot 16     & 18   &  16   & 15   &           \\ \hline
 24\cdot 8     & 2    &  2    & 2    &  2        \\ \hline
 8\cdot 8      & 2    &  2    & 2    &  2        \\ \hline
\end{array}\ \ \
\begin{array}{|r|r|r|r|r|}
\hline
 6\cdot 8      & 10   &  10   & 10   &  3        \\ \hline
 4'\cdot 8     & 10   &  10   & 10   &  2        \\ \hline
 3\cdot 8      & 1    &  1    & 1    &           \\ \hline
 2'\cdot 8     & 58   &  58   & 57   &  2        \\ \hline
 2''\cdot 8     & 7    &  7    & 7    &  1        \\ \hline
 1\cdot 8      & 71   &  65   & 63   &           \\ \hline
 12\cdot 4     & 1    &  1    & 0    &           \\ \hline
 6\cdot 4      & 3    &  3    & 3    &           \\ \hline
 4'\cdot 4     & 18   &  16   & 16   &           \\ \hline
 3\cdot 4      & 1    &  1    & 0    &           \\ \hline
 2'\cdot 4     & 173  &  169  & 169  &           \\ \hline
 2''\cdot 4     & 6    &  6    & 6    &           \\ \hline
 1\cdot 4      & 220  &  212  & 207  &           \\ \hline
 24\cdot 2     & 1    &  1    & 1    &           \\ \hline
 6\cdot 2      & 12   &  10   & 10   &           \\ \hline
 4'\cdot 2     & 38   &  30   & 30   &           \\ \hline
 3\cdot 2      & 7    &  3    & 1    &           \\ \hline
 2'\cdot 2     & 297  &  279  & 273  &           \\ \hline
 2''\cdot 2     & 46   &  40   & 36   &           \\ \hline
 1\cdot 2      & 1057 &  917  & 875  &           \\ \hline
 8\cdot 1      & 4    &  0    & 0    &           \\ \hline
 6\cdot 1      & 21   &  19   & 19   &           \\ \hline
 4'\cdot 1     & 16   &  0    & 0    &           \\ \hline
 4^\circ\cdot 1     & 12   &  0    & 0    &           \\ \hline
 3\cdot 1      & 16   &  10   & 6    &           \\ \hline
 2'\cdot 1     & 728  &  506  & 506  &           \\ \hline
 2''\cdot 1     & 78   &  0    & 0    &           \\ \hline
 1\cdot 1      & 5862 &  3018 & 2781 &           \\ \hhline{|=|=|=|=|=|}
\mbox{total}: & 8895 & 5511 & 5200 & 50   \\ \hline
\end{array}$
\caption{\footnotesize The table reflects the number of nonequivalent double-MDS-codes in $H(n,4)$, $n=4$, with the given number of automorphisms.
In the first column, the number of automorphisms of a double-MDS-code $C$ is represented in the form $P\cdot T$,
where $T=|\mathrm{Atop}(C)|$ and $P$ is the order of the group $\mathrm{Aut}(C)/\mathrm{Atop}(C)$ acting on the four coordinates. The accents reflect the type of this group:
groups of type $2'$ and $4'$ contain a transposition;
groups of type $2''$ and $4''$ consist of the identity permutation
and involutions without fixed points;
group of type $4^\circ$ corresponds to the cyclic group of order $4$;
all other groups of permutations of the four coordinated
are defined by their orders up to conjugancy.
The second column contains the number $N$
of nonequivalent double-MDS-codes with the corresponding
restrictions on the automorphism set.
The third column contains the number $N'$ of such codes
that are equivalent to their complement.
The fourth column contains the number $N''$
of such codes that are isotopic to their complement.
$N^*$ in the fifth column is the number
of splittable double-MDS-codes.
Using these data, the following numbers can be found for each row.
The number of isotopy classes of double-MDS-codes
is $N\cdot n!/P$.
The number of nonequivalent
frequency $n$-cubes F$^n(4;2,2)$ is $(N+N')/2$;
and $(N-N')/2$ of them have $P\cdot T$ automorphisms,
while the remaining $N'$ have twice more
(including the automorphisms that change
the value of the function).
The number of non-isotopic
frequency $n$-cubes F$^n(4;2,2)$ is $(N+N')\cdot n!/2P$;
and $(N+N'-2N'')\cdot n!/2P$ of them have $T$ autotopies,
while the remaining $N''\cdot n!/P$ have twice more.
}\label{t:4}
\end{table}

\begin{table}[htbp]
\centering
$\begin{array}{|r|r|r|r|r|}
\hline
 |\mathrm{Aut}(C)|   &N &N'&N''&N^*\\
\hhline{|=|=|=|=|=|}
 6 \cdot 256   & 1    &  1    & 1    &  1        \\ \hline
 2 \cdot 64    & 1    &  1    & 1    &  1        \\ \hline
 6 \cdot 32    & 1    &  1    & 1    &  1        \\ \hline
 2 \cdot 32    & 1    &  1    & 1    &  1        \\ \hline
 2 \cdot 16    & 1    &  1    & 1    &  1        \\ \hline
\end{array}\ \ \
\begin{array}{|r|r|r|r|r|}
\hline
 2\cdot 8     & 1   &  1   & 1   &  0        \\ \hline
 1\cdot 8     & 1   &  1   & 1   &  0        \\ \hline
 6\cdot 4     & 1   &  1   & 1   &  0        \\ \hline
 3\cdot 4     & 1   &  1   & 0   &  0        \\ \hline
 2\cdot 2     & 1   &  1   & 1   &  0        \\ \hhline{|=|=|=|=|=|}
\mbox{total}: & 10 & 10 & 9 & 5   \\ \hline
\end{array}$
\caption{\footnotesize
The table reflects the number of nonequivalent
double-MDS-codes in $H(n,4)$, $n=3$, with the given number
of automorphisms.
The notation and calculation of the derived values
are the same as those in Table~\ref{t:4}.
} \label{t:3}
\end{table}


\subsection[Double-MDS-codes in H(4,4)]{Splittability of double-MDS-codes in $H(4,4)$}

From the definitions, we can get the following relationship between the splittability
of a double-MDS-code and the splittability of each of its layers. If a double-MDS-code is splittable, then each layer of this code is splittable. If there exists a non-splittable layer, then the entire code is non-splittable. But if all layers are splittable, it does not necessarily mean
that the code itself is splittable.
\par
As a result of the classification, we found that all double-MDS-codes in $H(4,4)$,
except for two equivalence classes,
are splittable if their layers are splittable.
{In the diagram below, the black vertices induce a cycle of length $9$}.
$$
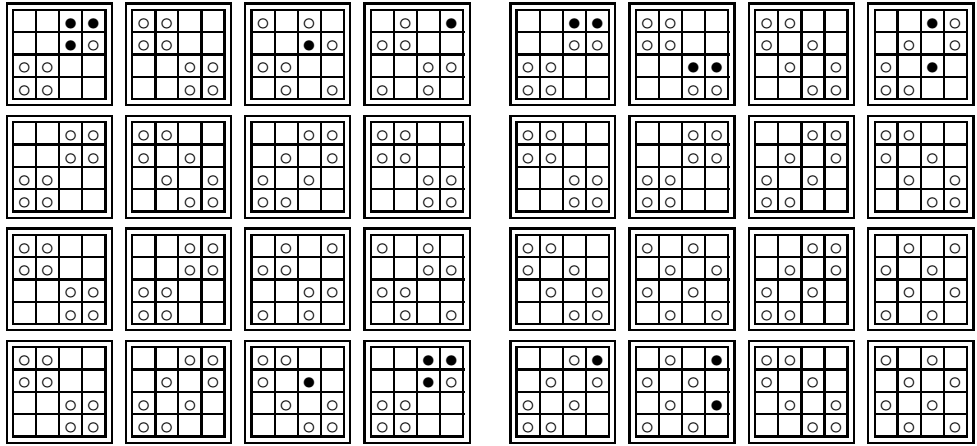
\def\2{\mbox{}\,\raisebox{-1mm}{$\bullet$}\,\mbox{}}
\def\1{\mbox{}\,\raisebox{-1mm}{$\circ$}\,\mbox{}}
\def\9{\mbox{}\phantom{\circ}\mbox{}}
\def\0{}
\def\HLINE{\hhline{~~||----||~||----||~||----||~||----||~||----||~||----||~||----||~||----||}}
\def\HHLINEt{\hhline{~~|t:====:t|~|t:====:t|~|t:====:t|~|t:====:t|~|t:====:t|~|t:====:t|~|t:====:t|~|t:====:t|}}
\def\HHLINEb{\hhline{~~|b:====:b|~|b:====:b|~|b:====:b|~|b:====:b|~|b:====:b|~|b:====:b|~|b:====:b|~|b:====:b|}}
\setlength{\extrarowheight}{-4mm} 
\begin{array}{@{}c@{}@{}c@{}||@{}c@{}|@{}c@{}|@{}c@{}|@{}c@{}||@{}c@{}||
                              @{}c@{}|@{}c@{}|@{}c@{}|@{}c@{}||@{}c@{}||
                              @{}c@{}|@{}c@{}|@{}c@{}|@{}c@{}||@{}c@{}||
                              @{}c@{}|@{}c@{}|@{}c@{}|@{}c@{}||@{}c@{~~~}||
                              @{}c@{}|@{}c@{}|@{}c@{}|@{}c@{}||@{}c@{}||
                              @{}c@{}|@{}c@{}|@{}c@{}|@{}c@{}||@{}c@{}||
                              @{}c@{}|@{}c@{}|@{}c@{}|@{}c@{}||@{}c@{}||
                              @{}c@{}|@{}c@{}|@{}c@{}|@{}c@{}||}
\HHLINEt
&&\0&\0&\2&\2&\9&\1&\1&\0&\0&\9&\1&\0&\1&\0&\9&\0&\1&\0&\2&\9& \0&\0&\2&\2&\9&\1&\1&\0&\0&\9&\1&\1&\0&\0&\9&\0&\0&\2&\1\\ \HLINE
&&\0&\0&\2&\1&\9&\1&\1&\0&\0&\9&\0&\0&\2&\1&\9&\1&\1&\0&\0&\9& \0&\0&\1&\1&\9&\1&\1&\0&\0&\9&\1&\0&\1&\0&\9&\0&\1&\0&\1\\ \HLINE
&&\1&\1&\0&\0&\9&\0&\0&\1&\1&\9&\1&\1&\0&\0&\9&\0&\0&\1&\1&\9& \1&\1&\0&\0&\9&\0&\0&\2&\2&\9&\0&\1&\0&\1&\9&\1&\0&\2&\0\\ \HLINE
&&\1&\1&\0&\0&\9&\0&\0&\1&\1&\9&\0&\1&\0&\1&\9&\1&\0&\1&\0&\9& \1&\1&\0&\0&\9&\0&\0&\1&\1&\9&\0&\0&\1&\1&\9&\1&\1&\0&\0\\ \HHLINEb
\\ \HHLINEt
&&\0&\0&\1&\1&\9&\1&\1&\0&\0&\9&\0&\0&\1&\1&\9&\1&\1&\0&\0&\9& \1&\1&\0&\0&\9&\0&\0&\1&\1&\9&\0&\0&\1&\1&\9&\1&\1&\0&\0\\ \HLINE
&&\0&\0&\1&\1&\9&\1&\0&\1&\0&\9&\0&\1&\0&\1&\9&\1&\1&\0&\0&\9& \1&\1&\0&\0&\9&\0&\0&\1&\1&\9&\0&\1&\0&\1&\9&\1&\0&\1&\0\\ \HLINE
&&\1&\1&\0&\0&\9&\0&\1&\0&\1&\9&\1&\0&\1&\0&\9&\0&\0&\1&\1&\9& \0&\0&\1&\1&\9&\1&\1&\0&\0&\9&\1&\0&\1&\0&\9&\0&\1&\0&\1\\ \HLINE
&&\1&\1&\0&\0&\9&\0&\0&\1&\1&\9&\1&\1&\0&\0&\9&\0&\0&\1&\1&\9& \0&\0&\1&\1&\9&\1&\1&\0&\0&\9&\1&\1&\0&\0&\9&\0&\0&\1&\1\\ \HHLINEb
\\ \HHLINEt
&&\1&\1&\0&\0&\9&\0&\0&\1&\1&\9&\0&\1&\0&\1&\9&\1&\0&\1&\0&\9& \1&\1&\0&\0&\9&\1&\0&\1&\0&\9&\0&\0&\1&\1&\9&\0&\1&\0&\1\\ \HLINE
&&\1&\1&\0&\0&\9&\0&\0&\1&\1&\9&\1&\1&\0&\0&\9&\0&\0&\1&\1&\9& \1&\0&\1&\0&\9&\0&\1&\0&\1&\9&\0&\1&\0&\1&\9&\1&\0&\1&\0\\ \HLINE
&&\0&\0&\1&\1&\9&\1&\1&\0&\0&\9&\0&\0&\1&\1&\9&\1&\1&\0&\0&\9& \0&\1&\0&\1&\9&\1&\0&\1&\0&\9&\1&\0&\1&\0&\9&\0&\1&\0&\1\\ \HLINE
&&\0&\0&\1&\1&\9&\1&\1&\0&\0&\9&\1&\0&\1&\0&\9&\0&\1&\0&\1&\9& \0&\0&\1&\1&\9&\0&\1&\0&\1&\9&\1&\1&\0&\0&\9&\1&\0&\1&\0\\ \HHLINEb
\\ \HHLINEt
&&\1&\1&\0&\0&\9&\0&\0&\1&\1&\9&\1&\1&\0&\0&\9&\0&\0&\2&\2&\9& \0&\0&\1&\2&\9&\0&\1&\0&\2&\9&\1&\1&\0&\0&\9&\1&\0&\1&\0\\ \HLINE
&&\1&\1&\0&\0&\9&\0&\1&\0&\1&\9&\1&\0&\2&\0&\9&\0&\0&\2&\1&\9& \0&\1&\0&\1&\9&\1&\0&\1&\0&\9&\1&\0&\1&\0&\9&\0&\1&\0&\1\\ \HLINE
&&\0&\0&\1&\1&\9&\1&\0&\1&\0&\9&\0&\1&\0&\1&\9&\1&\1&\0&\0&\9& \1&\0&\1&\0&\9&\0&\1&\0&\2&\9&\0&\1&\0&\1&\9&\1&\0&\1&\0\\ \HLINE
&&\0&\0&\1&\1&\9&\1&\1&\0&\0&\9&\0&\0&\1&\1&\9&\1&\1&\0&\0&\9& \1&\1&\0&\0&\9&\1&\0&\1&\0&\9&\0&\0&\1&\1&\9&\0&\1&\0&\1\\ \HHLINEb
\end{array}
$$
%
%

In the next section, we generalize found examples
by constructing, for every length $n\ge 3$, 
a non-splittable double-MDS-code
whose layers are all splittable.

\section{Non-splittable double-MDS-codes with splittable layers}\label{s:split}
The theorem proved in this section is important for us to understand
the interrelation between the classes of double-MDS-codes and splittable double-MDS-codes.
A double-MDS-code is defined by a local property; a code is a double-MDS-code
if and only if each of its layers in each direction is a double-MDS-code.
As we will see below, the similar is not true for splittable double-MDS-codes.

We first introduce some notations.
Two adjacent vertices of the Hamming graph differ in one coordinate,
whose number is called the \emph{direction} of the corresponding edge.
Edges of the same direction are called \emph{parallel}.
We color all edges of $H(n,4)$ by three colors depending on the values
of the two words in the coordinate
in which they are different.
If the values are either $0$ and $1$ or $2$ and $3$, then the color is $1$;
if the values are either $0$ and $2$ or $1$ and $3$,  the color is $2$;
for $0$, $3$ or $1$, $2$,  the color is $3$.
If a double-MDS-code is not splittable, then its induced subgraph
of the Hamming graph is not bipartite.
In this case, there is an odd-length cycle in this subgraph.
\begin{lemma}\label{l:odd}
 Any odd-length cycle in $H(n,4)$ contains
 three parallel edges of three different colors.
\end{lemma}
\begin{proof}
 If a cycle in $H(n,4)$ contains edges of at most two colors of some direction,
 then the number of the edges of this direction in the cycle is even.
 For example, if a cycle only has edges of color $2$ and $3$ in direction $i$,
 then each edge of these direction changes the value of the $i$th coordinate between the sets
 $\{0,1\}$ and $\{2,3\}$. Walking along all of the edges of the cycle, we have an
 even number of such changes because finally we return to the starting vertex.

 For a cycle of odd length, there is an odd number of edges of some direction.
 Hence, the edges of these directions are colored by three different colors.
\end{proof}

\begin{theorem}\label{th:odd}
 For any $n$ larger than $2$, there is an unsplittable double-MDS-code in $H(n,4)$ whose all layers
 in all directions are splittable.
\end{theorem}
\begin{proof}
The idea of the construction is based on Lemma~\ref{l:odd}.
We will construct an unsplittable double-MDS-code,
with an induced cycle of odd length $2n+1$,
such that the edges of the induced graph satisfy 
the following:
\begin{itemize}
 \item [(i)] the edges of any direction from $1$ to $n-1$ are colored by only two colors;
 \item [(ii)] for any two edges of direction $n$ and colors $1$ and $3$,
 the distance between them is $n-1$ (i.e., two words from different edges differ in each coordinate from $1$ to $n-1$).
\end{itemize}
Property (i) guarantees that every layer in the $n$th direction is splittable by Lemma~\ref{l:odd}.
On the other hand, Property (ii) guarantees that every layer in any other direction is splittable,
again by Lemma~\ref{l:odd}.

To construct such a code, we divide the vertex set of $H(n,4)$
into $2^{n-1}$ \emph{sectors} $V_a$, $a\in \{0,2\}^{n-1}\times\{0\}$,
where $V_a = \{ a+b \mid b \in \{0,1\}^{n-1}\times\{0,1,2,3\} \}$.
Define three $\{0,1\}$-valued functions $\alpha$, $\beta$, $\gamma$ on $\{0,1,2,3\}$:
$$
\begin{array}{r|cccc}
 x & 0 & 1 & 2 & 3 \\ \hline
 \alpha(x) & 1 & 1 & 0 & 0 \\
  \beta(x) & 1 & 0 & 1 & 0 \\
 \gamma(x) & 0 & 1 & 1 & 0 \\
\end{array}
$$
On $\{0,1,2,3\}^n$, define the following function:
$$
f(x)=
\begin{cases}
 x_1+ ... +x_{n-1} + \alpha(x_n) \bmod 2 & \mbox{ if $x\in V_{(0, ... , 0 ,0)}$,} \\
 x_1+ ... +x_{n-1} + \beta(x_n) + 1 \bmod 2 & \mbox{ if $x\in V_{(2, ... , 2, 0, ..., 0 ,0)}$ ($n-2$ sectors),} \\
 x_1+ ... +x_{n-1} + \gamma(x_n) \bmod 2 & \mbox{ if $x\in V_{(2, ... , 2 ,0)}$,} \\
 x_1+ ... +x_{n-1} + \beta(x_n)  \bmod 2 & \mbox{ otherwise}. \\
\end{cases}
$$
Denote by $M$ the set of ones of $f$.
It is straightforward to see that $f$ is an F$^n(4;2,2)$ frequency hypercube,
and so $M$ is a double-MDS-code.
Any edge of $H(n,4)$ of color $1$ and direction from $1$ to $n-1$ is in one sector,
and the values of $f$ on the two vertices of the edge are different. Hence, such an edge does not lie in $M$,
and (i) is satisfied.
Further, any edge in $M$ of color $1$ and direction  $n$ lies in the sector $V_{(0, ... , 0 ,0)}$,
while any edge in $M$ of color $3$ and direction  $n$ lies in  $V_{(2, ... , 2 ,0)}$.
Hence, (ii) is also satisfied. As was mentioned above, properties (i) and (ii) imply that any layer of $M$ is splittable.
It remains to find a cycle of odd length
\begin{itemize}
 \item
\underline{$(0, ... , 0, 0)$, $(0, ... , 0, 1)$},
 \item
$(2,0,  ...,0,1)$, 
$(2,2,0,...,0,1)$, \ldots, 
$(2,  ...,2,0,1)$, ($n-2$ vertices),
 \item
\underline{$(2, ... , 2, 1)$, $(2, ... , 2, 2)$},
 \item
 \underline{%
$(0,2,  ...,2,2)$, 
$(0,2,  ...,2,0)$},
$(0,0,2,...,2,0)$, \ldots, 
$(0,  ...,0,2,0)$,  ($n-1$ vertices)
\end{itemize}
(for convenience, we group vertices according to the four
cases in the definition of $f$).
The cycle has $3$ edges of direction $n$ (underlined)
and two edges of each other direction.
\end{proof}

\section{An upper bound on the number of double-MDS-codes}\label{sec:5}

The number of double-MDS-codes of length $n$ has a trivial upper bound $2^{3^n}$ because every such code can be reconstructed from only $3$ of $4$ layers in any direction. Thus, the size of the minimum testing set is no more than $3^n$.
For $n=3$, this value is $27$.
\par

For a subset $P$ of $\Sigma^{n}$, we denote by $X_P$
the column-vector 
$(x_{{0}},\ldots,x_{{4^n-1}})^\mathrm{T} = 
(x_{(0,...,0,0)},x_{(0,...,0,1)},...,x_{(3,...,3,3)})^T$,
over the Galois field GF$(2)$ of order $2$,
where $x_{(t_1,...,t_n)}=1$ if $(t_1,...,t_n)\in P$ and 
$x_{(t_1,...,t_n)}=0$ otherwise 
(the elements $(t_1,...,t_n)$ of $\Sigma^n$ in the indices are listed 
in the lexicographic order). 

We recursively define the matrix $A$ of size $n4^{n-1}\times 4^n$ over GF$(2)$ as follows:
$$
A=\left( \begin{array}{c} A_{1}\\  \vdots\\ A_{n} \end{array} \right) ,
\qquad
A_{i}=
I_{4^{i-1}}\otimes (1,1,1,1) \otimes I_{4^{n-i}}, \quad 1\leq i\leq n.
$$
Here and below, the symbol $\otimes$ denotes the Kronecker product of matrices, and $I_m$ is the $m \times m$ identity matrix.
The submatrices $A_{i}$ of $A$ have sizes $4^{n-1}\times 4^n$.

{\lemma\label{xxx} A subset
$D$ of $\Sigma^n$ is a unitrade if and only if 
$ A X_D = 0$, where $A$ is the matrix defined above.
}
\begin{proof}
 The claim is straightforward from the definition of unitrade.
 Indeed, each equation from the system $ A X_D = 0$ sums 
 the values of $X_D$ over the four elements of one line;
 this sum equals $0$ if and only if $X$ has an even number of elements
 in that line (the rows of the submatrix $A_i$ correspond to the lines of direction $i$).
For example, for a unitrade of length $2$,
the definition of a unitrade says that
$$\sum\limits_{i=0}^{3} x_{(i,j)}=0, \  j=0,1,2,3,  \quad \text{and} \quad
\sum\limits_{j=0}^{3} x_{(i,j)}=0,\ i=0,1,2,3,$$
which is the same as $ A X_D = 0$, where
$A
= \left(\begin{array}{c}
I_4\otimes(1,1,1,1)\\
(1,1,1,1)\otimes I_4
\end{array}
\right).$
\end{proof}

{\lemma\label{l:3n} The number of unitrades in $\Sigma^n$ is $2^{3^n}$.
In particular, the rank of the matrix $A$ is $4^n-3^n$.
}
\begin{proof}
We firstly note that each element of $\Sigma^n$ with at $m$, $m>0$, zeros
belongs to a line where each of the other three elements 
has $m-1$ zeros. It follows that the elements of a unitrade with a zero component are uniquely reconstructed from its elements without zero components.
That is, a unitrade is uniqiely defined by its intersection with $\{1,2,3\}^n$.
So, the number of unitrades is not larger than the number $2^{3^n}$ of subsets
of $\{1,2,3\}^n$.

It remains to find ${3^n}$ linearly independent solutions of $ A X_D = 0$.
For every $a$ from $\{1,2,3\}^n$, let the set $D_a$ consist of all
$2^n$ elements of $\Sigma^n$ obtained from $a$ by replacing some 
(maybe all, maybe none) of components by zeros. It is easy to see
that $D_a$ is a unitrade. Moreover, all $X_{D_a}$, $a\in\{1,2,3\}^n$,
are linearly independent because each element of $\{1,2,3\}^n$
belongs to a unique ${D_a}$.
\end{proof}

\par

The following proposition allows to construct testing sets for double-MDS-codes.

\begin{proposition}\label{xxxx}
Suppose that there exists unitrade $D$ in $H(n,4)$ 
such that no two double-MDS-codes $C_1$, $C_2$
satisfy $\emptyset\ne C_1 \triangle C_2 \subset D$.
Then for double-MDS-codes of length $n$,
there is a testing set $T\subset \Sigma^n \setminus D$
of size $3^n-k_D$,
where $2^{k_D}$ is the number of unitrade subsets of $D$, ${k_D}\ge 1$.
\end{proposition}

\begin{proof}
Denote   $\overline{D}=\Sigma^n\backslash {D}$
and
$N=|\overline{D}|$. Let
$\overline{D}=\{{u}_0,{u}_1,...,{u}_{N-1}\}$. Note that
$\overline{D}$ is also a unitrade.
By $B$, we denote the $N\times 4^n$ matrix
with the rows
$X_{\{{u}_0\}}$, \ldots, $X_{\{{u}_{N-1}\}}$.
%
The solution space of
$\Big(\!\!\begin{array}{c}A\\B\end{array}\!\!\Big)X=0$ 
over  GF$(2)$ is
$\{X_P \mid P\subset D,\text{ $P$ is a unitrde}\}$.
 The dimension of this space is $k_D$, because
$ 2^{k_D}= | \{P \mid P\subset D,\text{ $P$ is a unitrde}\} |$
by the definition of $k_D$.
Hence, the rank of the matrix 
$\Big(\!\!\begin{array}{c}A\\B\end{array}\!\!\Big)$
is $4^n - k_D$.
The rank of $A$ is $4^n-3^n$ by Lemma~\ref{l:3n}.
So, we can find  $3^n-k_D$ rows of $B$, with numbers $i_1$, \ldots, $i_{3^n-k_D}$, forming a
$(3^n-k_D) \times 4^n$
submatrix $B'$ such that the rank
of $\Big(\!\!\begin{array}{c}A\\B'\end{array}\!\!\Big)$ is also $4^n - k_D$.
Each row of $B'$ has $1$ in exactly one position, corresponding to some $ u_i$.
We state that the 
subset $T=\{  u_i \mid i= i_1,  \ldots, i_{3^n-k_D} \}$
of $\overline D$
is a testing set for double-MDS-codes.
Indeed, if two double-MDS-codes
$C_1$ and $C_2$ meet
$C_1\cap T=C_2\cap T$,
then $C=C_1\triangle C_2$
does not intersect with $T$.
Hence, the vector $X_C$
is in the kernel of $B'$, i.e., $B'X_C = 0$.
Moreover, 
$C$ is a unitrade because it is 
the symmetric difference of unitrades.
Therefore, $X_C$ is also in the kernel of $A$, 
i.e., $AX_C = 0$.
We conclude that
$\Big(\!\!\begin{array}{c}A \\ B'\end{array}\!\!\Big) X_C = 0$, 
and hence
$\Big(\!\!\begin{array}{c}A\\B\end{array}\!\!\Big) X_C = 0$ 
(indeed, the row spaces of the matrices 
$\Big(\!\!\begin{array}{c}A\\ B'\end{array}\!\!\Big) $ 
and 
$\Big(\!\!\begin{array}{c}A\\B\end{array}\!\!\Big) $ 
coincide).
The last equality means that
$C$ is in $ \{P \mid P\subset D,\text{ $P$ is a unitrde}\} $.
By the hypothesis
of the proposition, $C$ is an empty set. It follows that $C_1=C_2$ and
$T$ is a testing set.
\end{proof}
\par
The straightforward check,
utilizing the classification in Theorems~\ref{th:M3}
and~\ref{th:U3}, shows that there are two nonequivalent unitrades in $H(n,4)$, of cardinalities $32$ and $38$,
that satisfy
the hypothesis of the proposition with $k_D=2$,
and there are no
unitrades with $k_D=3$.
The two examples are shown at the following 
diagram where white and black bullets indicate
 unitrade subsets $D'$ and $D''$, so $2^{k_D}=|\{\emptyset,D',D'',D\}|$.
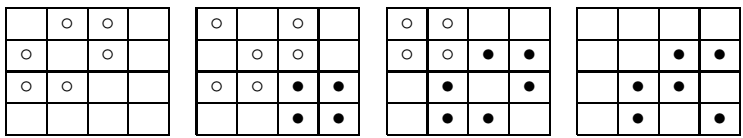
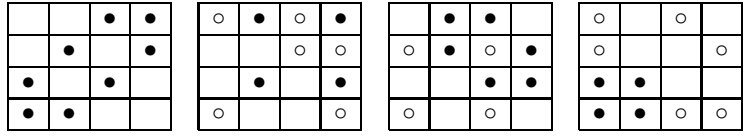
$$
\def\1{\bullet}
\def\2{\color{black}\circ}
\def\0{\phantom{\circ}}
\def\HLINE{\cline{1-4}\cline{6-9}\cline{11-14}\cline{16-19}}
\begin{array}{|c|c|c|c|c|c|c|c|c|c|c|c|c|c|c|c|c|c|c|}
\HLINE
  \0&\2&\2&\0& & \2&\0&\2&\0& & \2&\2&\0&\0  & &   \0&\0&\0&\0   \\\HLINE
  \2&\0&\2&\0& & \0&\2&\2&\0& & \2&\2&\1&\1  & &   \0&\0&\1&\1   \\\HLINE
  \2&\2&\0&\0& & \2&\2&\1&\1& & \0&\1&\0&\1  & &   \0&\1&\1&\0   \\\HLINE
  \0&\0&\0&\0& & \0&\0&\1&\1& & \0&\1&\1&\0  & &   \0&\1&\0&\1   \\\HLINE
\end{array}
$$
$$
\def\1{\bullet}
\def\2{\color{black}\circ}
\def\0{\phantom{\circ}}
\def\HLINE{\cline{1-4}\cline{6-9}\cline{11-14}\cline{16-19}}
\begin{array}{|c|c|c|c|c|c|c|c|c|c|c|c|c|c|c|c|c|c|c|}
\HLINE
\0&\0&\1&\1& &\2&\1&\2&\1& &\0&\1&\1&\0& &\2&\0&\2&\0\\\HLINE
\0&\1&\0&\1& &\0&\0&\2&\2& &\2&\1&\2&\1& &\2&\0&\0&\2\\\HLINE
\1&\0&\1&\0& &\0&\1&\0&\1& &\0&\0&\1&\1& &\1&\1&\0&\0\\\HLINE
\1&\1&\0&\0& &\2&\0&\0&\2& &\2&\0&\2&\0& &\1&\1&\2&\2\\\HLINE
\end{array}
$$
The second unitrade has the following complement,
$$
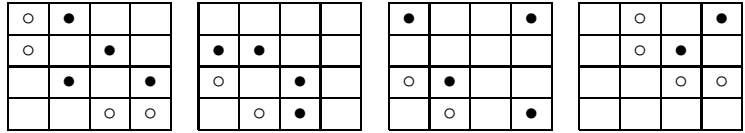
\def\1{\bullet}
\def\2{\color{black}\circ}
\def\0{\phantom{\circ}}
\def\HLINE{\cline{1-4}\cline{6-9}\cline{11-14}\cline{16-19}}
\begin{array}{|c|c|c|c|c|c|c|c|c|c|c|c|c|c|c|c|c|c|c|}
                                                       \HLINE
\2&\1&\0&\0& &\0&\0&\0&\0& &\1&\0&\0&\1& &\0&\2&\0&\1\\\HLINE
\2&\0&\1&\0& &\1&\1&\0&\0& &\0&\0&\0&\0& &\0&\2&\1&\0\\\HLINE
\0&\1&\0&\1& &\2&\0&\1&\0& &\2&\1&\0&\0& &\0&\0&\2&\2\\\HLINE
\0&\0&\2&\2& &\0&\2&\1&\0& &\0&\2&\0&\1& &\0&\0&\0&\0\\\HLINE
\end{array}
$$
where removing any white vertex  results in a testing set of size $25$.

This result is generalized by the following lemma, which is a special case of~{\cite[Proposition~26]{KroPot:3}}
(for completeness, we give a proof).
\begin{lemma}\label{l:3l}
Let $T\subset\Sigma^{3}$ be a testing set for double-MDS-codes of length $3$.
Then the Cartesian product $T^{l}\subset\Sigma^{3l}$ of the testing set $T$ is a testing set for double-MDS-codes of length $3l$.
\end{lemma}
\begin{proof}
We prove the claim by induction. Suppose that for two double-MDS-codes $C_1$ and $C_2$ we have $C_{1}\mid_{T^{l}}=C_{2}\mid_{T^{l}}$. Then, by the induction hypothesis, for any $c\in T$, $C_{1}\mid_{T^{l-1}\times\{c\}}=C_{2}\mid_{T^{l-1}\times\{c\}}$ implies $C_{1}\mid_{\Sigma^{3(l-1)}\times\{c\}}=C_{2}\mid_{\Sigma^{3(l-1)}\times\{c\}}$. Hence, for any $\omega\in \Sigma^{3(l-1)}$, we have $C_{1}\mid_{\{\omega\}\times T}=C_{2}\mid_{\{\omega\}\times T}$. The set $\{\omega\}\times T$ is testing on $\{\omega\}\times\Sigma^{3}$. Then $C_{1}\mid_{\{\omega\}\times\Sigma^{3}}=C_{2}\mid_{\{\omega\}\times\Sigma^{3}}$ for any $\omega\in \Sigma^{3(l-1)}$.
\end{proof}
Substituting the testing set of size $25$ found above
to Lemma~\ref{l:3l}, we get a testing set of size
$25^l$ for double-MDS-codes in $H(3l,4)$.
\begin{corollary}\label{c:1}
The number of double-MDS-codes in $\Sigma^{n}$,
where $n$ is divisible by $3$, is at most $2^{\beta^{n}}$,
where $\beta =25^{\frac{1}{3}}<3$.
\end{corollary}
With a 	slightly worth bound,
the result is generalized to an arbitrary $n\ge 3$.
\begin{theorem}\textup{(upper bound)}
The number of $N_n$ double-MDS-codes in $\Sigma^{n}$, $n\ge 3$, 
is at most $2^{\alpha^{n}}$, where $\alpha<2.955$. 
Moreover, for $n$ sufficiently large, we have $N_n<2^{2.925^{n}}$.
\end{theorem}
\begin{proof}
By Corollary~\ref{c:1},
for $n=3l$, $l=1,2,\ldots$,
the size of a testing set is no more than $25^l=\alpha_n^n$,
where 
$$\alpha_n = \alpha_{3l}  \equiv \beta, 
\qquad \beta=\sqrt[3]{25} < 2.925.$$

For $n=3l+1$, a double-MDS-code can be considered as the union of $4$ layers,
$4$ double-MDS-codes of length $3l$.
Moreover, each of them is uniquely determined by the other three layers.
Therefore, the minimum size of a testing set
is no more than $25^l\cdot 3=\alpha_n^n$ ,
where 
$$\alpha_n 
=\alpha_{3l+1} 
=(25^l\cdot 3)^{\frac{1}{n}}
= \textstyle\beta\cdot \exp \frac{\ln3-\ln\beta}n.
$$

Similarly, for $n=3l+2$ the size of a testing set is bounded by
$25^l\cdot 3\cdot 3=\alpha_n^n$, where
$$\alpha_n 
=\alpha_{3l+2} 
=(25^l\cdot 3^2)^{\frac{1}{n}}
=\textstyle\beta \cdot \exp \frac{2(\ln3-\ln\beta)}n.
$$

Since the sequences $\alpha_{3l}$, $\alpha_{3l+1}$, $\alpha_{3l+2}$
are non-increasing, the first claim
of the theorem holds with
$$ \alpha 
= \max_{n\ge 3} \alpha_n
= \max \{\alpha_3,\alpha_4,\alpha_5\}
= \max\{25^{\frac{1}{3}}, 75^{\frac{1}{4}}, 225^{\frac{1}{5}}\}  = 225^{\frac{1}{5}} <2.955.$$

Finally, we see that $\alpha_n \to \beta$, 
which proves the second claim.
\end{proof}

\begin{remark}
It is worth to check if the same approach works for more general
class of objects, e.g. for double-codes. Direct computations show
that the set of symmetric differences of double-codes in $H(3,4)$
is the the set of all unitrades, and a testing set of size less than $27$ cannot
be found in the way described. So, the best upper bound on the number of double-codes remains trivial, i.e., $2^{3^n}$.
\end{remark}

\section{Acknowledgement}
The authors are grateful to Vladimir Potapov for useful discussions.


\providecommand\href[2]{#2} \providecommand\url[1]{\href{#1}{#1}}
  \def\DOI#1{{\small {DOI}:
  \href{http://dx.doi.org/#1}{#1}}}\def\DOIURL#1#2{{\small{DOI}:
  \href{http://dx.doi.org/#2}{#1}}}

\end{document}